\documentclass[12pt]{amsart}
\usepackage{cmap} 
\usepackage[T2A]{fontenc}   
\usepackage[utf8]{inputenc}
\usepackage[english,russian]{babel}

\usepackage{amsmath}
\usepackage{amssymb,color}
\usepackage[mathcal]{eucal}
\usepackage[pdftex]{hyperref} 
\usepackage{orcidlink}

\newtheorem{deff}{Definition}[section]

\newtheorem{theorem}[deff]{Theorem}
\newtheorem{coro}[deff]{Corollary}

\newtheorem{prop}[deff]{Proposition}

\newtheorem{em-example}[deff]{Example}
\newtheorem{em-def}[deff]{Definition}        
\newtheorem{em-remark}[deff]{Remark}         
\newtheorem{em-question}[deff]{Question}

\newtheorem{problem}[deff]{Problem}

\newenvironment{example}{\begin{em-example} \em }{ \end{em-example}}

\newcommand{\pcal}{\mathcal {P}}

\def\emp{\emptyset}

\def\sm{\setminus}
\def\sub{\subseteq}

\def\om{\omega}

\DeclareMathSymbol{\res}{\mathord}{AMSa}{"16}

\def\:{\nobreak \hskip .1111em\mathpunct {}\nonscript \mkern
   -\thinmuskip {:}\hskip .3333emplus.0555em\relax}

\catcode`\@=12

\def\R{{\mathbb R}}

\def\cont{\mathfrak c}

\def\ii{{\mathfrak{i}}}
\def\:{\colon}

\begin{document}

\title[Regular rigid Korovin orbits]{Regular rigid Korovin orbits}
\author{Evgenij Reznichenko and Mikhail Tkachenko$\,^{1,\,2}$\orcidlink{0000-0002-1698-381X}}
\address{Department of General Topology and Geometry, 
Mechanics and Mathematics Faculty, M.~V.~Lomonosov Moscow 
State University, Leninskie Gory 1, Moscow, 199991 Russia}
\email{erezn@inbox.ru}
\address{Departamento de Matem\'aticas, Universidad Aut\'onoma Metropolitana, 
Av. San Rafael Atlixco 186, Col. Vicentina, Del. Iztapalapa, C.P. 09340, 
Mexico City, Mexico}
\email{mich@xanum.uam.mx}
\keywords{Korovin orbit; feebly compact; $\mathbb{R}$-rigid; $C$-embedded; 
semitopological group; quasitopological group}
\subjclass[2020]{22A30, 54H15 (primary), 54A25, 54C45 (secondary)}

\thanks{$^1\,$Corresponding author}
\thanks{$^2\,$This author was supported by grant number CAR-64356 
of the Program \lq\lq{Ciencia de Frontera 2019\rq\rq} of the CONACyT, Mexico.}

\date{January 31, 2025}

\selectlanguage{english}
\begin{abstract}
An example of an infinite regular feebly compact quasitopological 
group is presented such that all continuous real-valued functions
on the group are constant. The example is based on the use of 
Korovin orbits in $X^G$, where $X$ is a special regular countably 
compact space constructed by S.~Bardyla and L.~Zdomskyy in 2023 
and $G$ is an abstract Abelian group of an appropriate cardinality. 
Also, we study the interplay between the separation properties of the 
space $X$ and Korovin orbits in $X^G$. We show in particular that 
if $X$ contains two nonempty disjoint open subsets, then every
Korovin orbit in $X^G$ is Hausdorff. 
\end{abstract}
\selectlanguage{russian}

\maketitle

\section{Introduction}
According to the celebrated theorem of T.~Banakh and A.~Ravsky in \cite{BR},
every regular paratopological group is Tychonoff, that is, completely regular. 
Actually, the same article provides a more general fact: \emph{Every regular
topological monoid with open shifts is completely regular.} It is also shown
there that every Hausdorff paratopological group is function\-al\-ly Hausdorff,
that is, continuous real-valued functions on the group separate points. 

These results suggest that one could extend the theorem about the 
coincidence of regularity and complete regularity to other objects of 
Topological Algebra, such as semitopological groups or the smaller 
class of quasitopological groups. 

We show in Theorem~\ref{Th:MEx} that this is not the case and that
there exists a regular \emph{quasitopological group} $H$ (that is, a 
group with a topology such that the left and right translations and inversion 
in the group are continuous), but every continuous real-valued function on
$H$ is constant, that is, $H$ is \emph{$\mathbb{R}$-rigid}. Hence, $H$ 
is neither completely regular nor functionally Hausdorff. Additionally,
the space $H$ is \emph{feebly compact}, that is, every locally finite 
family of open sets in $H$ is finite. 

The group $H$ in Theorem~\ref{Th:MEx} is a \emph{Korovin orbit}
(see Section~\ref{sec:korovin} for more details). Every Korovin orbit 
is a subspace of a product space $X^G$, for some space $X$ and 
an abstract Abelian group $G$, that fills all countable subproducts 
in $X^G$.

The relationships between the separation properties of a product 
space $X=\prod_{i\in I} X_i$ and a dense subspace $Y$ of $X$ 
that fills all finite (or countable) subproducts of $X$ are studied in 
Section~\ref{sec:ssp}. We also provide some conditions on $Y$ 
and the factors $X_i$ implying that $Y$ is either $C$-embedded 
in $X$ or $\mathbb{R}$-rigid. The results obtained there are applied 
in Section~\ref{sec:korovin} for the study of the same problems 
for Korovin orbits $G_f$ in place of $Y$. 

In the very short Section~\ref{Sec:2}, we prove the aforementioned
Theorem~\ref{Th:MEx} by applying results from Sections~\ref{sec:ssp}
and~\ref{sec:korovin}. Four open problems in Section~\ref{Sec:OP} 
show a further direction of the study of dense subspaces of uncountable 
topological products and have a close connection with the contents of 
the article.

\section{Notation, terminology and preliminary facts}\label{SS1.1}

\subsection{Subspaces of topological products}\label{sec:d:ssp}
Given a topological product $X=\prod_{i\in I} X_i$ and a nonempty
subset $J$ of the index set $I$, we denote by $\pi_J$ the projection
of $X$ onto the subproduct $X_J=\prod_{i\in J} X_i$. A subset $Y$ 
of $X$ \emph{fills countable (finite) subproducts in $X$} if $\pi_J(Y)=X_J$, 
for each countable (finite) set $J\sub I$. Evidently, a set $Y\sub X$
that fills finite (countable) subproducts of $X$ is dense in $X$. 
\smallskip

The proof of the subsequent fact requires a standard application
of the $\Delta$-lemma (see Theorem on page~87 of \cite{Juh}) 
and, hence, is omitted. 

\begin{prop}\label{Le:TopPrN}
Let $X=\prod_{i\in I} X_i$ be a topological product such that 
the subproduct $X_J=\prod_{i\in J}X_i$ is feebly compact, 
for each countable set $J\sub I$. Then $X$ is also feebly 
compact. Furthermore, if a subspace $D$ of $X$ fills all
countable subproducts of $X$, then $D$ is also feebly compact.
\end{prop}

A space is said to be \emph{pseudo-$\aleph_1$-compact} 
if every locally finite family of open sets in the space is countable. 
The following result is a special case of Theorem on page 90 of 
\cite{CN} if one takes $\varkappa=\om$ and $\alpha=\om_1$:

\begin{prop}\label{cor.fc.last}
Let $D$ be a subspace of a product $X=\prod_{i\in I} X_i$ that fills
all countable subproducts of $X$. Let also $g\colon D\to Z$ be a 
continuous mapping to a metrizable space $Z$. If the space $X$ 
is pseudo-$\aleph_1$-compact, then so is $D$ and one can find a 
countable set $J\sub I$ and a continuous mapping $h\colon\pi_J(D)\to Z$ 
such that $g=h\circ\pi_J\res{D}$.
\end{prop}

\subsection{Axioms of separation}\label{sec:as}
Let $X$ be a space and $i\in \{0,1,2,3,3.5,4\}$.
We say that $X$ is a $T_i$-space if it satisfies condition $(T_i)$:
\begin{enumerate}
\item[$(T_{0})$]  Given distinct points $x,y\in X$, the exists an open 
                           set $U$ in $X$ that contains exactly one of the points $x,y$.

\item[$(T_{1})$]  Given distinct points $x,y\in X$, the exists an open 
                           set $U$ in $X$ such that $x\in U$ and $y\notin U$.
                          
\item[$(T_{2})$]  The Hausdorff separation property (two distinct points 
                           in $X$ have disjoint open neighborhoods).
                           
\item[$(T_{3})$] If $U$ is an open neighborhood of a point $x\in X$, then
                          there exists an open neighborhood $V$ of $x$ such that
                          $\overline{V}\subset U$.

\item[$(T_{3.5})$] If a closed set $F\sub X$ does not contain a point 
                             $x\in X$, then there exists a continuous real-valued 
                             function $f$ on $X$ satisfying $f(x)=1$ and $f(F)\sub \{0\}$.

\item[$(T_{4})$]  Every two closed disjoint subsets of $X$ have disjoint
                           open neigh\-bor\-hoods in $X$.
\end{enumerate}

Hausdorff spaces are the spaces satisfying condition $(T_2)$.
Regular spaces are defined to satisfy conditions $(T_3)$ and $(T_1)$.
The conjunction $(T_{3.5})\, \&\, (T_1)$ defines completely regular 
(Tychonoff) spaces, whereas the combination $(T_4)\, \&\,  (T_1)$ 
delineates normal spaces. \emph{A functionally Hausdorff} space 
$X$ is a space in which continuous real-valued functions separate 
points of $X$.

\subsection{Groups with a topology}\label{sec:tgroups}
A group with a topology is called \emph{semitopological} if multiplication 
in the group is separately continuous. Ellis \cite{Ellis1957-2} proved that 
a locally compact semitopological group is a topo\-logical group, that is, 
multiplication and inversion in the group are continuous. In his work 
\cite{Korovin1992}, Korovin demonstrated that a completely regular, 
countably compact Abelian semitopological group is a topo\-log\-ical group. 
Additionally, he provided an example of a pseudocompact semitopological 
Abelian group with discontinuous multiplication. In \cite{Reznichenko1994}, 
Korovin's theorem is extended to the non-Abelian case. 

In the same article \cite{Korovin1992}, a wide class of pseudocompact 
spaces was introduced that were afterwards called \emph{Korovin spaces} 
in \cite{Reznichenko2024gr,ReznichenkoTkachenko2024}. This class 
contains countably compact spaces, pseudocompact spaces of count\-able
tightness and some other important subclasses of pseudocompact
spaces. If a semitopological group $G$ is a Korovin space, then 
$G$ is a topological group \cite{Reznichenko1994} (in the case 
$G$ is Abelian, this is a Korovin theorem in \cite{Korovin1992}). 
A detailed study of the class of Korovin spaces is presented in
\cite{Reznichenko2024gr}. 

A semitopological group $G$ is called \emph{quasitopological} 
if inversion in the group is continuous. Every boolean semitopological 
group is evidently a quasitopological group.

\subsection{Korovin orbits}\label{sec:d:korovin}
In \cite{Korovin1992}, A.~Korovin presented an original method to constructing 
pseudocompact semitopological groups that are not topological groups.
An analysis of Korovin's method led to the notions of \emph{Korovin map} 
and \emph{Korovin orbit} \cite{ArhangelskiiHusek2001,HernandezTkachenko2006,ReznichenkoTkachenko2024} 
(see also Theorem~2.4.13 in \cite{AT}). Korovin orbits are also used to construct pseudocompact quasitopo\-logical groups with various additional properties \cite{ArhangelskiiHusek2001,HernandezTkachenko2006,Batikova2009,TangLinXuan2021}.
Nontrivial Korovin orbits cannot be Korovin spaces. 

Let $X$ be a space with $|X|>1$ and $G$ be an infinite Abelian 
group. The group $G$ naturally acts on the product space $X^G$.
The action is defined by the formula $(gf)(x)=f(xh)$ for all $g,x\in G$ 
and $f\in X^G$. We denote by $G_f=Gf=\{gf: g\in G\}$ the 
\emph{orbit} of an element $f\in X^G$ under the action of $G$ 
on $X^G$. The element $f$ is called a \emph{Korovin mapping} 
if $G_f$ fills in all countable subproducts in $X^G$, $\pi_M(G_f)=X^M$ 
for each countable set $M\sub G$. In this case, $G_f$ is called 
a \emph{Korovin orbit.}

A sequence $(g_n)_{n\in\om}\subset G$ is said to be \emph{exact}
if $g_k\neq g_n$ for distinct $k,n\in\om$. A mapping $f\colon G\to X$
is a Korovin mapping if and only if the following condition holds:
\begin{enumerate}
\item[$(K_\om)$] If $(g_n)_{n\in\om}\subset G$ is an exact sequence 
and $(x_n)_{n\in\om}\subset X$, then there exists $g\in G$ such that
$f(g_ng)=x_n$ for each $n\in\om$.
\end{enumerate}

\begin{theorem}[$\mathbf{A.\,V.~Korovin}$]\label{Kor-pse-tm}
Let $\kappa$ be an infinite cardinal with $\kappa=\kappa^\om$, 
$X$ be a topological space satisfying $1<|X|\leq\kappa$, and 
$G$ an Abelian group such that $|G|=\kappa$. Then there exists 
a Korovin mapping $f\colon G\to X$ and, hence, there exists a 
Korovin orbit in $X^G$.
\end{theorem}

It is easy to see that the mapping $\ii\: G\to G_f$, $g\mapsto gf$ 
is a bijection \cite[Proposition~2.4.14{\hskip1pt}(1)]{AT}. This
bijection is used to define multiplication in $G_f$ making use 
of multiplication in the group $G$ as follows: $\ii(g)\ii(h)=\ii(gh)$ for 
all $g,h\in G$. Hence, the groups $G$ and $G_f$ are algebraically
isomorphic. When $G_f$ carries the topology of a subspace of $X^G$, 
this multiplication turns $G_f$ into a semitopological group (see 
\cite{Korovin1992} or \cite[Proposition~2.4.14{\hskip1pt}(3)]{AT}). 
In what follows we will sometimes identify $G$ with the corresponding 
Korovin orbit $G_f$. If the group $G$ is boolean, then $G_f$ is a 
quasitopological group.

In \cite{ReznichenkoTkachenko2024}, the following quite 
unexpected properties of Korovin orbits are established:

\begin{theorem}[{\cite[Theorem 1]{ReznichenkoTkachenko2024}}]\label{t:korovin:1}
	Let $G$ be an Abelian group, $X$ be a space with $|X|\geq 2$, and 
	$f\colon G\to X$  a Korovin mapping.
	   \begin{enumerate}
			\item
			If $X$ is not anti-discrete, then every countable subset 
			of the Korovin orbit $G_f$ is closed and discrete in $G_f$.
			\item
			If there exists a non-constant continuous real-valued function 
			on $X$, then every countable subset in the Korovin orbit $G_f$ 
			is $C^\ast$-embedded in $G_f$.
			\item
			If there exists an unbounded continuous real-valued function 
			on $X$, then any countable subset of the Korovin orbit $G_f$ 
			is $C$-embedded in $G_f$.
	\end{enumerate}
\end{theorem}

\section{Subspaces of topological products}\label{sec:ssp}
Let  $X=\prod_{i\in I} X_i$ is a topological product and $Y$ be 
a subspace of $X$. Our aim is to find relationships between
the separation properties of the space $Y$ and those of the
factors $X_i$ provided $Y$ fills the finite (countable) subproducts
of $X$.

Assume that $\pcal$ is a class of spaces closed under arbitrary 
topological products and taking subspaces. Clearly, if $X_i\in\pcal$ 
for each $i\in I$, then $X\in\pcal$ and $Y\in\pcal$. Therefore, the 
subsequent statement is immediate.

\begin{prop}\label{p:ssp:1}
Let $\pcal$ be one of the following eight classes:
$T_i$-spaces for some $i\in\{0,1,2,3,3.5\};$ 
functionally Hausdorff spaces; regular spaces; 
completely regular spaces. If $X_i\in\pcal$ for
each $i\in I$, then the product space $X=\prod_{i\in I} X_i$ 
as well as every subspace $Y$ of $X$ are also in $\pcal$.
\end{prop}

\begin{prop}\label{p:ssp:2}
Suppose that a subspace $Y$ of $X=\prod_{i\in I} X_i$ fills all 
finite subproducts of $X$. If{\hskip1pt} $Y$ is a $T_3$-space, 
then each factor $X_i$ is also a $T_3$-space.
\end{prop}

\selectlanguage{english}
\begin{proof}
Assume that the subspace $Y$ of $X$ satisfies the $T_3$
separation axiom. Take any index $i_0\in I$ and let $U$ be an 
open neighborhood of an arbitrary point $a\in X_{i_0}$. Choose 
a point $x\in Y$ with $\pi_{i_0}(x)=a$, where $\pi_{i_0}\colon 
X\to X_{i_0}$ is the projection, $\pi_{i_0}(y)=y_{i_0}$ for every 
$y\in X$. Clearly, $x$ is an element of the open set $V=
\pi_{i_0}^{-1}(U)$. Since $Y$ is a $T_3$-space, there exists 
a canonical open set $W=\prod_{i\in I} W_i$ in $X$ such that 
$x\in W$ and $cl_{Y}(W\cap Y)\sub V\cap Y$ or, by the density 
of $Y$ in $X$, $Y\cap cl_{X}W\sub Y\cap V$. Let $F=
\{i\in I: W_i\neq X_i\}$. Diminishing the set $W$, if necessary, 
we can assume that $i_0\in F$. Notice that
$a=\pi_{i_0}(x)\in \pi_{i_0}(W)=W_{i_0}$.

Let us verify that $cl_X W_{i_0}\sub U$. Suppose for a 
contradiction that there exists a point $b\in (cl_X W_{i_0})
\sm U$. Take any point $c\in P=\prod_{i\in F} cl_{X_i} W_i$
such that $c_{i_0}=b$. Since $\pi_F(Y)=X_F$, there exists 
a point $y\in Y$ such that $\pi_F(y)=c$. Clearly, $y_{i_0}=
c_{i_0}=b$. Notice that $W=\pi_F^{-1}\pi_F(W)$ and $cl_{X} W=
\pi_F^{-1} \big(cl_{X^F}\pi_F(W)\big)=\pi_F^{-1}(P)$. Therefore, 
$y\in \pi_F^{-1}(P)\cap Y=(cl_{X} W)\cap Y$. Since
$(cl_{X} W)\cap Y\sub V$, we see that $y\in V=
\pi_{i_0}^{-1}(U)$ and, hence, $b=\pi_{i_0}(y)\in U$.
This contradiction proves our claim and implies that 
$X_{i_0}$ is a $T_3$-space.
\end{proof}

An analogue of Proposition~\ref{p:ssp:1} is valid for $T_{3.5}$ 
separation property. 

\begin{prop}\label{p:ssp:2+1}
Suppose that a subspace $Y$ of $X=\prod_{i\in I} X_i$ fills 
all finite subproducts of $X$ and is $C^\ast$-embedded in $X$. 
If $Y$ is a $T_{3.5}$-space, then each factor $X_i$ is also a 
$T_{3.5}$-space.
\end{prop}

\begin{proof}
Let $i_0\in I$ and $U$ be an open neighborhood of a point 
$x_0\in X_{i_0}$. Take an element $a\in Y$ such that 
$\pi_{i_0}(a)=x_0$, where $\pi_{i_0}\colon X\to X_{i_0}$ 
is the projection. Then $V=\pi_{i_0}^{-1}(U)\cap Y$ is an open 
neighborhood of $a$ in $Y$. By our assumption, $Y$ is a 
$T_{3.5}$-space, so there exists a continuous function $g$ 
on $Y$ with values in the closed unit interval $[0,1]$ such 
that $g(a)=1$ and $g(Y\sm V)=\{0\}$. Let $\tilde{g}$ be an
extension of $g$ to a continuous function on $X$. The
point $a$ has the form $a=(x_0,y_0)$, where $y_0\in \prod
_{i\in I\sm\{i_0\}} X_i=P$. We define a continuous real-valued 
function $h$ on $X_{i_0}$ by letting $h(x)=\tilde{g}(x,y_0)$ for 
each $x\in X_{i_0}$. Then $h(x_0)=\tilde{g}(x_0,y_0)=g(a)=1$ and, 
similarly, $h(x)=\tilde{g}(x,y_0)=0$ for each $x\in X\sm U$, because 
$(x,y_0)\notin V$ in this case and the set $Y\cap \pi_{i_0}^{-1}(x)$ 
is dense in $\pi_{i_0}^{-1}(x)=\{x\}\times P$. We have thus shown 
that the function $h$ separates the point $x_0$ from the closed 
set $X_{i_0}\sm U$, so $X_{i_0}$ is a $T_{3.5}$-space. 
\end{proof}

It is important to note that the implications presented in 
Propositions~\ref{p:ssp:2} and~\ref{p:ssp:2+1} are, in fact, 
equivalencies, according to Proposition~\ref{p:ssp:1}. 
Furthermore, in Example~\ref{Exa:D}, it will be shown that 
neither of the two previous propositions holds true for regularity 
or complete regularity when substituting the $T_3$ or $T_{3.5}$ 
separation properties, respectively. 

Let us demonstrate that $\mathbb{R}$-rigidity behaves similarly to  
$T_3$ and $T_{3.5}$ separation properties. 

\begin{prop}\label{p:ssp:4}
Assume that a space $X_i$ is $\R$-rigid for each $i\in I$. Then 
$X=\prod_{i\in I} X_i$ is $\R$-rigid. Furthermore, if a subspace 
$Y\sub X$ is $C^\ast$-embedded in $X$, then $Y$ is also $\R$-rigid.
\end{prop}

\begin{proof}
By induction, one can easily show that the product of finitely many 
$\mathbb{R}$-rigid spaces is $\mathbb{R}$-rigid. Hence, every 
$\sigma$-product $Z$ in $X$ is also $\mathbb{R}$-rigid. Since 
$Z$ is dense in $X$, the product space $X$ is $\mathbb{R}$-rigid  
 as well. 

Suppose for a contradiction that $f$ is a non-constant continuous 
real-valued function on a $C^\ast$-embedded subspace $Y$ of 
$X$. Take points $a,b\in Y$ such that $f(a)\neq f(b)$ and let 
$M=\max\{|f(a)|,|f(b)|\}$. Clearly, $M>0$. Then the function 
$f_M$ on $Y$ defined by 
\begin{equation*}
f_M(x)=\begin{cases}f(x)&\text{if $|f(x)|\leq M$};\\
-M&\text{if $f(x)< -M$};\\ 
M&\text{if $f(x)>M$} 
\end{cases}
\end{equation*}
is continuous and bounded by $M$. Let $g$ be an extension of 
$f_M$ to a continuous function on $X$. Then $g$ is constant, 
because $X$ is $\mathbb{R}$-rigid. Consequently, $f_M$ is 
also constant. However, the definition of $f_M$ implies that
$f(a)=f_M(a)=g(a)$ and $f(b)=f_M(b)=g(b)$, so $f(a)=f(b)$. This 
contradiction proves that the space $Y$ is $\mathbb{R}$-rigid.
\end{proof}

One can complement Proposition~\ref{p:ssp:4} by noting that
if{\hskip1pt} $Y$ is an $\mathbb{R}$-rigid subspace of $X=\prod_{i\in I}X_i$
and the projection $\pi_{i_0}(Y)$ of $Y$ to the factor $X_{i_0}$
is dense in $X_{i_0}$ for some $i_0\in I$, then the space $X_{i_0}$ 
is also $\mathbb{R}$-rigid. This is a simple consequence of the
fact that a continuous image of an $\mathbb{R}$-rigid space is
also $\mathbb{R}$-rigid. Hence, if $Y$ is a dense $\mathbb{R}$-rigid
subspace of $X$, then each factor $X_i$ is $\mathbb{R}$-rigid.

\begin{prop}\label{p:ssp:3}
Suppose that $Y$ is a subspace of $X=\prod_{i\in I} X_i$ that
fills all countable subproducts in $X$ and satisfies one of the 
following conditions:
\begin{enumerate}
\item[\rm{(a)}] the space $X_J=\prod_{i\in J} X_i$ is pseudo-$\aleph_1$-compact  
                       for any countable set $J\subset I$; 
\item[\rm{(b)}] $X_i$ is first-countable for each $i\in I$;
\item[\rm{(c)}] $X_i$ is separable for each $i\in I$.
\end{enumerate}
Then $Y$ is $C$-embedded in $X$.
\end{prop}

\begin{proof}
If $X_J$ is pseudo-$\aleph_1$-compact  for any countable $J\sub I$, 
then $X$ are pseudo-$\aleph_1$-compact\,---\,this is a simple consequence
of the $\Delta$-lemma. Hence, every continuous real-valued function 
on $Y$ extends to a continuous function on $X$. Indeed, let 
$g\colon Y\to \mathbb{R}$ be a continuous function. By 
Proposition~\ref{cor.fc.last}, we can find a countable subset $J$ of
$I$ and a continuous function $h$ on $X_J$ such that $g=h\circ
\pi_J\res{Y}$. Let $\tilde{g}=h\circ\pi_J$. The function $\tilde{g}$
on $X$ is continuous and extends $g$, as claimed. 

Assume that the space $X_i$ is first-countable for any $i\in I$. 
Then \cite[Corollary~1.6.16]{AT} implies that $X$ is a Moscow 
space. Since $Y$ fills all countable subproducts in $X$, it follows 
that $Y$ meets every nonempty $G_\delta$-set in $X$. Applying 
\cite[Theorem~6.1.7]{AT} we conclude that $Y$ is 
$C$-embedded in $X$.  

Also, if (c) holds, then the space $X_J$ is separable for any 
countable set $J\subset I$, hence, pseudo-$\aleph_1$-compact. 
So (c) is a special case of (a). 
\end{proof}

The following result is immediate from Propositions~\ref{p:ssp:4} 
and~\ref{p:ssp:3}.

\begin{coro}\label{c:ssp:1}
Let $X_i$ be a separable $\R$-rigid space for each $i\in I$. 
If a subspace $Y$ of $X=\prod_{i\in I} X_i$ fills all countable
subproducts of $X$, then $Y$ is $\R$-rigid.
\end{coro}

\section{Korovin orbits. Axioms of separation in Korovin orbits}\label{sec:korovin}
Let $X$ be a nontrivial space, $G$ be an infinite Abelian group,
and $f\colon G\to X$ a Korovin mapping. Clearly, the existence 
of $f$ implies certain restrictions on the cardinalities of $X$ and
$G$ (see Theorem~\ref{Kor-pse-tm}). Our aim is to explore the 
relationship between the separation properties of a space $X$ 
and those of a Korovin orbit $G_f \sub X^G$.

\begin{prop}\label{Prop:T1}
If a space $X$ is not anti-discrete, then every Korovin orbit 
$G_f$ in $X^G$ is a $T_1$-space.
\end{prop}
\begin{proof}
By  item (1) of Theorem~\ref{t:korovin:1}, every singleton in 
$G_f$ is a closed set.
\end{proof}

\begin{prop}\label{Prop:Hau}
Every (equivalently, some) Korovin orbit $G_f$ in $X^G$ is Hausdorff 
if and only if the space $X$ contains two disjoint nonempty open sets.
\end{prop}

\begin{proof}
Let $U$ and $V$ be nonempty disjoint open sets in $X$. Take
distinct elements $u$ and $v$ of a Korovin orbit $G_f\sub X^G$. There 
exist $g_0,g_1\in G$ such that $u=g_0f$ and $v=g_1f$. Choose elements 
$g_n\in G$ for $n>1$ such that the sequence $(g_n)_{n\in\om}$ is exact. 

Pick points $x_0\in U$ and $x_1\in V$. For every $n>1$, we put 
$x_n=x_1$. Condition $(K_\om)$ implies that there exists $g\in G$ such 
that $f(g_ng)=x_n$ for each $n\in\om$. Then $u(g)=(g_0f)(g)=f(gg_0)=x_0$ 
and $v(g)=(g_1f)(g)=f(gg_1)=x_1$. Hence, $G_f\cap\pi_g^{-1}(U)$ and
 $G_f\cap \pi_g^{-1}(V)$ are disjoint open neighborhoods of the points 
 $u$ and $v$ in $G_f$, where $\pi_g\colon X^G\to X_{(g)}$ is the
 projection.

Conversely, assume that a Korovin orbit $G_f$ in $X^G$ is Hausdorff. 
Take distinct elements $u,v\in G_f$ and choose disjoint open neighborhoods
$O_u$ and $O_v$ of $u$ and $v$, respectively, in $G_f$. Then one 
can find canonical open neighborhoods $W_u$ and $W_v$ of $u$ and 
$v$ in $X^G$ such that $W_u\cap G_f\sub O_u$ and $W_v\cap G_f
\sub O_v$. Let $W_u=\prod_{g\in G} W_{u,g}$ and $W_v=\prod_{g\in G}
W_{v,g}$, where $W_{u,g}$ and $W_{v,g}$ are open in $X$ for each
$g\in G$. Let $F$ be a finite subset of $G$ such that $W_{u,g}=W_{v,g}=X$
for all $i\in I\sm F$. If $W_{u,g}\cap W_{v,g}\neq\emp$ for each $g\in F$,
then $W_u\cap W_v$ is a nonempty open set in $X^G$, so the density
of $G_f$ in $X^G$ implies that
$$
\emp\neq W_u\cap W_v\cap G_f = (W_u\cap G_f)\cap (W_v\cap G_f) 
\sub O_u\cap O_v=\emp.
$$ 
This contradiction shows that there exists $g\in F$ such that the nonempty 
open subsets $W_{u,g}$ and $W_{v,g}$ of the space $X$ are disjoint. 
\end{proof}

\begin{prop}\label{Prop:FH}
If a space $X$ admits a non-constant real-valued function, then 
every Korovin orbit $G_f$ in $X^G$ is functionally Hausdorff.
\end{prop}

\begin{proof}
Let $u$ and $v$ be distinct elements of $G_f$. It follows from
item (2) of Theorem~\ref{t:korovin:1} that there exists a continuous
real-valued function $q$ on $G_f$ such that $q(u)=0$ and $q(v)=1$. 
\end{proof}

\begin{prop}\label{Prop:Reg}
A Korovin orbit $G_f$ is a $T_3$-space if and
only if $X$ is a $T_3$-space. 
\end{prop}

\begin{proof}
Since $G_f$ fills all countable subproducts in $X^G$, 
the required conclusion follows directly from 
Propositions~\ref{p:ssp:1} and~\ref{p:ssp:2}.
\end{proof}

The case of $T_{3.5}$ separation axiom is considered in the
following result.

\begin{prop}\label{Prop:CoReg}
Let $f\colon G\to X$ be a Korovin mapping and $G_f\sub X^G$ 
the corresponding Korovin orbit. Suppose that one of the following
holds true:
\begin{enumerate}
\item[\rm{(a)}] the space $X^\om$ is pseudo-$\aleph_1$-compact; 
\item[\rm{(b)}] $X$ is first-countable;
\item[\rm{(c)}] $X$ is separable.
\end{enumerate}
Then $G_f$ is a $T_{3.5}$-space if and only if $X$ is a $T_{3.5}$-space. 
\end{prop}

\begin{proof}
Since $X$ is a $T_{3.5}$-space, so are the spaces $X^G$ and 
$G_f\sub X^G$. Hence, it suffices to verify the converse implication. 
We assume, therefore, that $G_f$ is a $T_{3.5}$-space. 

Since $G_f$ fills all countable subproducts in $X^G$, 
Proposition~\ref{p:ssp:3} implies that $G_f$ is $C$-embedded 
in $X^G$. Then we apply Proposition~\ref{p:ssp:2+1} to conclude 
that $X$ is a $T_{3.5}$-space. 
\end{proof}

\begin{prop}\label{Prop:zd}
If a space $X$ is not anti-discrete and has a base of clopen sets,
then every Korovin orbit $G_f$ in $X^G$ is a regular zero-dimensional
space.
\end{prop}

\begin{proof}
Since $X$ has a base of clopen sets, the same is valid for the 
spaces $X^G$ and $G_f$. Also, $X$ is a $T_1$-space, by 
Proposition~\ref{Prop:T1}. It follows that the space $G_f$
is regular and zero-dimensional.
\end{proof}

The subsequent example shows that the conclusion of 
Proposition~\ref{Prop:CoReg} does not hold without 
additional assumptions on $X$. 

\begin{example}\label{Exa:D}
Let $X$ be the product of the discrete space $D=\{0,1\}$ and 
the same set $D$ that carries the anti-discrete topology. Also,
denote by $G$ the boolean group $\mathbb{Z}(2)^\om$. Then,
by  Proposition~\ref{Prop:zd}, every Korovin orbit $G_f$ in $X^G$ 
is completely regular, but $X$ is not even a $T_0$-space. 
\end{example}

Together, Corollory~\ref{c:ssp:1} and the fact that a Korovin orbit 
$G_f$ fills all countable subproducts in $X^G$ imply the following.

\begin{coro}\label{c:korovin:1}
Let $X$ be a separable $\R$-rigid space. Then every Korovin
orbit $G_f$ in $X^G$ is $\R$-rigid.
\end{coro}

\section{An example}\label{Sec:2}
The results presented in Sections~\ref{SS1.1} and~\ref{sec:korovin} as  
well as the existence of an infinite regular separable countably compact 
$\mathbb{R}$-rigid space \cite{BaZd} enable us to give a short proof of 
the following theorem announced in the abstract.

\begin{theorem}\label{Th:MEx}
There exists an infinite regular feebly compact $\R$-rigid 
quasitopological group $H$. Therefore, $H$ is neither completely 
regular nor functionally Hausdorff.
\end{theorem}

\begin{proof}
Let $X$ be a regular separable, \emph{totally countably compact} 
$\R$-rigid space with $|X|=2^\cont$, where $\cont=2^\om$ (see 
\cite[Theorem~3.4]{BaZd}). Total countable compactness of $X$ 
means that every infinite set $A\sub X$ contains an infinite 
subset $B$ such that the closure of $B$ in $X$ is compact. The 
latter property of $X$ implies that $X^\om$ is also totally countably 
compact. In particular, $X^\om$ is feebly compact.

Let also $G=\mathbb{Z}(2)^\cont$, where $\mathbb{Z}(2)$ is the 
Boolean group $\{0,1\}$. By Theorem~\ref{Kor-pse-tm}, there 
exists a Korovin mapping $f\colon G\to X$. Denote by $H=G_f$ 
the corresponding Korovin orbit considered as a subspace of $X^G$. 

Since $X^\om$ is countably compact, Proposition~\ref{Le:TopPrN} 
implies that the space $G_f$ is feebly compact. The regularity of $X$ 
guarantees that the spaces $X^G$ and $G_f$ are also regular. By 
Corollary~\ref{c:korovin:1}, $G_f$ is $\R$-rigid.
\end{proof}

\section{Open problems}\label{Sec:OP}
This article's results raise several unanswered questions.

Let us call a Korovin orbit $G_f\sub X^G$ \emph{nontrivial} 
if $|X|\geq 2$ and $|G|\geq 2^\om$.

\begin{problem}\label{Pr:1}
Can a non-trivial Korovin orbit be a normal space?
\end{problem}

\begin{problem}\label{Pr:2}
Can a non-trivial Korovin orbit be a paracompact space?
\end{problem}

The next problem is a more general version of the previous one.

\begin{problem}\label{Pr:3}
Let a subspace $D$ of $X=\prod_{i\in I}X_i$ fills countable subproducts 
of $X$, where $|I|>\om$ and each $X_i$ is not compact. Can $D$ be
paracompact?
\end{problem}

\begin{problem}\label{Pr:4}
Let a subspace $D$ of $X=\prod_{i\in I} X_i$ fills all countable (finite) 
subproducts of $X$ and assume that $X$ is a $T_1$-space.
\begin{enumerate}
\item[{\rm (1)}] If $D$ is paracompact, is each factor $X_i$ 
                       also paracompact?

\item[{\rm (2)}] If $D$ normal, is each factor $X_i$ normal?

\item[{\rm (3)}] If $D$ is completely regular, is each factor 
                       $X_i$ completely regular?
\end{enumerate}
\end{problem}

The group $H=G_f$ in Theorem~\ref{Th:MEx} is not separable because 
it is uncountable and all countable subsets of $H$ are closed (see
item (1) of Theorem~\ref{t:korovin:1}). For this reason, we inquire 
as to whether separability can be added to the properties of the 
group $H$:

\begin{problem}\label{Pr:5}
Is there a regular separable semitopological (quasitopological)
group that is not completely regular?
\end{problem}


\end{document}